\documentclass{llncs}
\usepackage[dvips]{graphicx}
\usepackage{amssymb}
\usepackage{amsmath}
\usepackage{dsfont}
\usepackage{yfonts}
\usepackage[T1]{fontenc}
\usepackage{algorithmic}
\usepackage{algorithm}
\usepackage{multirow}
\usepackage{psfrag}
\newcommand{\F}{\mathbb{F}}
\newcommand{\Q}{\mathbb{Q}}
\newcommand{\N}{\mathbb{N}}
\newcommand{\Z}{\mathbb{Z}}

\newcommand{\C}{\mathbb{C}}

\newcommand{\End}{\textrm{End}}

\begin{document}

\title{Pairing-based algorithms for jacobians of genus 2 curves with maximal endomorphism ring}
\author{Sorina Ionica}
\institute{
Ecole Normale Sup\'erieure\footnote{This work was carried during the
  author's stay at the Ecole Polytechnique, team TANC and
  at LORIA, Nancy, team CARAMEL.}\\
45 Rue d'Ulm, Paris, 75005, France\\
\email{sorina.ionica@m4x.org}
}
\maketitle

\begin{abstract}
Using Galois cohomology, Schmoyer characterizes cryptographic non-trivial self-pairings of the $\ell$-Tate pairing in terms of the action of the Frobenius on the $\ell$-torsion of the Jacobian of a genus 2 curve. We apply similar techniques to study the non-degeneracy of the $\ell$-Tate pairing restrained to subgroups of the $\ell$-torsion which are maximal isotropic with respect to the Weil pairing. First, we deduce a criterion to verify whether the jacobian of a genus 2 curve has maximal endomorphism ring. Secondly, we derive a method to construct horizontal $(\ell,\ell)$-isogenies starting from a jacobian with maximal endomorphism ring.
\end{abstract}
\section{Introduction}

A central problem in elliptic and hyperelliptic curve cryptography is that of constructing an elliptic curve or an abelian surface
having a given number of points on their Jacobian. The solution to this problem relies on the computation of the Hilbert class polynomial
for a quadratic imaginary field in the genus one case. The analogous genus 2 case needs the Igusa class polynomials for quartic CM fields.
There are three different methods to compute these polynomials: an analytic algorithm~\cite{Weng}, a $p$-adic algorithm~\cite{GauHou} and a Chinese Remainder Theorem-based algorithm~\cite{EisLau}. The last one relies heavily on an algorithm for determining endomorphism rings of the jacobians of genus 2 curves over prime fields.
Eisentr\"ager and Lauter~\cite{EisLau} gave the first algorithm for computing endomorphism rings of Jacobians of genus 2 curves over finite fields. The algorithm takes as input a jacobian $J$ over a finite field and a primitive quartic CM field $K$, i.e. a purely imaginary quadratic extension field of a real quadratic field with no proper imaginary quadratic fields. The real quadratic subfield $K_0$ has class number $1$. The main idea is to compute a set of generators of an order $\mathcal{O}$ in the CM field and then to test whether these generators are endomorphisms of $J$, in order to decide whether the order $\mathcal{O}$ is the endomorphism ring $\textrm{End}(J)$ or not. In view of application to the CRT method for Igusa class polynomial computation, Freeman and Lauter bring a series of improvements to this algorithm, in the particular case where we need to decide whether $\textrm{End}(J)$ is the maximal order or not.
Note that the Eisentr\"ager-Lauter CRT method for class polynomial
computation searches for curves defined over some prime field $\F_p$ and
belonging to a certain isogeny class. Once such a curve is found, the
algorithm keeps the curve only if it has maximal endomorphism ring. This
search is rather expensive and ends only when all curves having maximal
endomorphism ring were found. Recent research in the
area~\cite{BelBro,Sutherland1,BroGru} has shown that we can significantly
reduce the time of this search by using \textit{horizontal isogenies},
i.e. isogenies between jacobians having the same endomorphism ring. Indeed,
once a Jacobian with maximal endomorphism ring is found, many others can be
generated from it by computing horizontal isogenies. In this paper, we
propose a new method for checking if the endomorphism ring is locally
maximal at $\ell$, for $\ell>2$ prime, and a method to compute kernels of
horizontal $(\ell,\ell)$-isogenies. Our methods rely on the computation of the Tate pairing. 

Let $H$ be a genus 2 smooth irreducible curve defined over a finite field $\F_q$, $J$ its jacobian and suppose that $J[\ell^n]\subseteq J(\F_q)$ and that $J[\ell^{n+1}]\nsubseteq J(\F_q)$, with $\ell$ different from $p$ and $n\geq 1$. We denote by $\mathcal{W}$ the set of rank 2 subgroups in
$J[\ell^n]$, which are isotropic with respect to the $\ell^n$-Weil pairing. We define $k_{\ell}$ to be
\begin{eqnarray*}
k_{\ell}=\max_{G\in \mathcal{W}}\{k|\exists P,Q\in G~\textrm{and}~T_{\ell^n}(P,Q)\in \mu_{\ell^k}\backslash \mu_{\ell^{k-1}}\}.
\end{eqnarray*}
The jacobian $J$ is ordinary, hence it has complex multiplication by an
order in a quartic CM field $K$. We assume that $K=\Q(\eta)$, with $\eta=i\sqrt{a+b\sqrt{d}}$ if $d\equiv 2,3\mod 4$ or $\eta=i\sqrt{a+b\left ( \frac{-1+\sqrt{d}}{2} \right )}$ if $d\equiv 1\mod 4$. We consider the decomposition of the Frobenius endomorphism $\pi$ over a basis of the ring of integers of $K$ : $\pi=a_1+a_2\frac{-1+\sqrt{d}}{2}+(a_3+a_4\frac{-1+\sqrt{d}}{2})\eta$, if $d\equiv 1 \bmod 4$ and $\pi=a_1+a_2\sqrt{d}+(a_3+a_4\sqrt{d})\eta$, if $d\equiv 2,3 \bmod 4$. We assume that the coefficients verify the following condition 
\begin{eqnarray}\label{condition}
\max (v_{\ell}(\frac{a_3-a_4}{\ell}),v_{\ell}(\frac{a_3-\ell a_4}{\ell^2})) <\min (v_{\ell}(a_3), v_{\ell}(a_4)).
\end{eqnarray}
We show that if condition~\eqref{condition} is satisfied, the computation
of $k_{\ell}$ suffices to check whether the endomorphism ring is locally
maximal at $\ell$, in many cases. Moreover, our method to distinguish
kernels of horizontal $(\ell,\ell)$-isogenies from other $(\ell,\ell)$-isogenies is also related to $k_{\ell}$. Given $G$ an element of $\mathcal{W}$, we say that the Tate pairing is $k_{\ell}$-non-degenerate (or simply non-degenerate) on $G\times G$ if the restriction map
\begin{eqnarray*}
T_{\ell^n}:G\times G\rightarrow \mu_{\ell^{k_{\ell}}}
\end{eqnarray*}
is surjective. Otherwise, we say that the Tate pairing is $k_{\ell}$-degenerate (or simply degenerate) on $G\times G$. Our main result is the following theorem.
\begin{theorem}\label{horizontal}
Let $H$ be a genus 2 smooth irreducible curve defined over a finite field $\F_q$ and $\ell>2$ a prime number. Let $J$ be the jacobian of $H$, whose endomorphism ring is a locally maximal order at $\ell$ of a CM-field $K$. Assume that the real quadratic subfield $K_0$ has class number 1. Suppose that the Frobenius endomorphism $\pi$ is such that $\pi-1$ is exactly divisible by $\ell^n$, $n\in \Z$ and that $k_{\ell}>0$. Let $G$ be a subgroup of rank 2 in $J[\ell]$ which is isotropic with respect to the Weil pairing. Let $\bar{G}$ be a rank 2 subgroup in $J[\ell^n]$ isotropic with respect to the $\ell^n$-Weil pairing and such that $\ell^{n-1}\bar{G}=G$. Then the following hold
\begin{enumerate}
\item If the isogeny of kernel $G$ is horizontal, then the Tate pairing is $k_{\ell}$-degenerate over $\bar{G}\times \bar{G}$.
\item If the condition~\eqref{condition} is satisfied and the Tate pairing is $k_{\ell}$-degenerate over $\bar{G}\times \bar{G}$, then the isogeny is horizontal. 
\end{enumerate}
\end{theorem}
In view of application to the CRT method for Igusa class polynomial computation, we deduce an algorithm to compute kernels of horizontal isogenies efficiently. This generalizes a result on horizontal $\ell$-isogenies for genus 1 curves~\cite{IonJou2}.

This paper is organised as follows. In Section~\ref{EisentragerLauter} we recall briefly the Eisentr\"ager-Lauter algorithm for computing endomorphism rings. In Section~\ref{TatePairing} we give the definition and properties of the Tate pairing. Section~\ref{MaximalOrders} describes our algorithm for checking whether a Jacobian has locally maximal order at $\ell$. In Section~\ref{HorizontalIsogenies} we show that we can compute kernels of horizontal $(\ell,\ell)$-isogenies by some Tate pairing calculations. Finally, Section~\ref{Complexity} gives complexity estimates for our algorithms and compares their performance to that of the Freeman-Lauter algorithm.

\paragraph*{Notation and assumptions.} In this paper, we assume that principally polarized abelian surfaces are \textit{simple}, i.e. not isogenous to a product of elliptic curves. A quartic CM field $K$ is a totally imaginary quadratic extension of a totally real field. We denote by $K_0$ the real quadratic subfield of $K$ and we assume that $K_0$ has class number 1.
 A CM-type $\Phi$ is a couple of pairwise non-complex conjugate embeddings of K in $\C$
\begin{eqnarray*}
\Phi(z)=(\phi_1(z),\phi_2(z)).
\end{eqnarray*}
An abelian surface over $\C$ with complex multiplication by $\mathcal{O}_K$ is given by $A(\C)=\C^2/\Phi(\mathfrak{a})$, where $\mathfrak{a}$ is an ideal of $\mathcal{O}_K$  and $\Phi$ is a CM type. This variety is said to be of CM-type $(K,\Phi)$. A CM-type $(K,\Phi)$ is primitive if $\Phi$ cannot be obtained as a lift of a CM-type of a CM-subfield of $K$. The principally polarized abelian variety $\C^2/\Phi(\mathfrak{a})$ is simple if and only if its CM-type is \textit{primitive}~\cite{Shimura}.

\section{Computing the endomorphism ring of a jacobian}\label{EisentragerLauter}

The endomorphism ring of an ordinary jacobian $J$ over a finite field $\F_q$ ($q=p^n$) is an order in a quartic CM field $K$ such that
\begin{eqnarray*}
\Z[\pi,\bar{\pi}]\subset \textrm{End}(J) \subset \mathcal{O}_K,
\end{eqnarray*}
where $\Z[\pi,\bar{\pi}]$ denotes the order generated by $\pi$, the Frobenius endomorphism and by $\bar{\pi}$, the Verschiebung. We give a brief description of the Eisentr\"ager-Lauter algorithm~\cite{EisLau} which computes the endomorphism ring of $J$.
For a fixed order $\mathcal{O}$ in the lattice of orders of $K$, the algorithm tests whether this order is contained in $\End(J)$.
This is done by computing a $\Z$-basis for the order and checking whether the elements of this basis are endomorphisms of $J$ or not.
In order to test if $\alpha \in \mathcal{O}$ is an endomorphism, we write
\begin{eqnarray}\label{above}
\alpha=\frac{a+b\pi+c\pi^2+d\pi^3}{n},
\end{eqnarray}
with $a,b,c,d, n$ some integers such that $a,b,c,d$ have no common factor with $n$ ($n$ is the smallest integer such that $n\alpha \in \Z[\pi]$). The LLL algorithm computes a sequence $a,b,c,d,n$ such that $\alpha$ can be written as in Equation~\ref{above}.
In order to check whether $\alpha$ is an endomorphism or not, Eisentr\"ager and Lauter~\cite{EisLau} use the following result.
\begin{lemma}\label{torsion}
Let $A$ be an abelian variety defined over a field $k$ and $n$ an integer coprime to the characteristic of $k$. Let $\alpha:A\rightarrow A$ be an endomorphism of $A$. Then $A[n]\subset \textrm{Ker}~\alpha$ if and only if there is another endomorphism $\beta$ of $A$ such that $\alpha=n\cdot \beta$.
\end{lemma}
Using Lemma~\ref{torsion}, we get $\alpha\in \End(J)$ if and only if $a+b\pi+c\pi^2+d\pi^3$ acts as zero on the $n$-torsion.
Freeman and Lauter show that $n$ divides the index $[\mathcal{O}_K:\Z[\pi]]$ (see~\cite[Lemma 3.3]{FreLau}). Since $[\Z[\pi]:\Z[\pi,\bar{\pi}]$ is $1$ or $p$, we have that $n$ divides $[\mathcal{O}_K:\Z[\pi,\bar{\pi}]]$ if $(n,p)=1$. Moreover, Freeman and Lauter show that if $n$ factors as $\ell_1^{d_1}\ell_2^{d_2}\ldots \ell_r^{d_r}$, it suffices to check if
\begin{eqnarray*}
\frac{a+b\pi+c\pi^2+d\pi^3}{\ell_i^{d_i}},
\end{eqnarray*}
for every prime factor $\ell_i$ in the factorization of $n$.
The advantage of using this family of elements instead of $\alpha$ is that instead of working over the extension field generated by the coordinates of the $n$-torsion points, we may work over the field of definition of the $\ell_i^{d_i}$-torsion, for every prime factor $\ell_i$.
For a fixed prime $\ell$, Freeman and Lauter prove the following result, which allows computing a bound for the degree of the smallest extension field over which the $\ell$-torsion points are defined.

\begin{proposition}\cite[Prop. 6.2]{FreLau}\label{extensionField}
Let $J$ be the Jacobian of a genus 2 curve over $\F_q$ and suppose that $\textrm{End}(J)$ is isomorphic to the ring of integers $\mathcal{O}_K$ of the primitive quartic CM field $K$. Let $\ell\neq q$ be a prime number, and suppose $\F_{p^r}$ is the smallest field over which the points of $J[\ell]$ are defined. If $\ell$ is unramified in $K$, then $r$ divides one of the following:
\begin{itemize}
\item[(a)] $\ell-1$, if $\ell$ splits completely in $K$;
\item[(b)] $\ell^2-1$, if $\ell$ splits into two or three ideals in $K$;
\item[(c)] $\ell^3-\ell^2+\ell-1$, if $\ell$ is inert in $K$.
\end{itemize}
If $\ell$ ramifies in $K$, then $r$ divides one of the following:
\begin{itemize}
\item[(a)] $\ell^3-\ell^2$, if there is a prime over $\ell$ of ramification degree $3$, or if $\ell$ is totally ramified in $K$ and $\ell\leq 3$.
\item[(b)] $\ell^2-\ell$, in all other cases where $\ell$ factors into four prime ideals in $K$ (counting multiplicities).
\item[(c)] $\ell^3-\ell$, if $\ell$ factors into two or three prime ideals in $K$ (counting multiplicities).
\end{itemize}
\end{proposition}
Once we computed the extension field over which the $\ell$-torsion is defined, the $\ell^d$-torsion will be computed using the following result~\cite{FreLau}.
\begin{proposition}\cite[Prop. 6.3]{FreLau}\label{moreTorsion}
Let $A$ be an ordinary abelian variety defined over a finite field $\F_q$ and let $\ell$ be a prime number not equal to the characteristic of $\F_q$. Let $d$ be a positive integer. If the $\ell$-torsion points of $A$ are defined over $\F_q$, then the $\ell^d$-torsion points are defined over $\F_{q^{\ell^{d-1}}}$.
\end{proposition}

\section{Background on the Tate pairing}\label{TatePairing}
Consider now $H$ a smooth irreducible genus 2 curve defined over a finite field $\F_q$, with $q=p^r$, whose equation is
\begin{equation}
y^2+h(x)y=f(x),
\end{equation}
with $h,f\in \F_q[x]$, $\textrm{deg}\,h\leq 2$, $f$ monic and $\deg f=5,6$.
Let $J$ be the jacobian of $H$ and denote by $\bar{\F}_q$ the algebraic closure of $\F_q$ and by $G_{\bar{\F}_q/\F_q}=\textrm{Gal}(\bar{\F}_q/{\F_q})$ the Galois group. Let $m\in \N$ and consider $J[m]$ the subgroup of $m$-torsion, i.e. the points of order $m$. We denote by $\mu_m\subset \bar{\F}_{q}$ the group of $m$-th roots of unity. The $m$-Weil pairing
\begin{eqnarray*}
W_m:J[m]\times \hat{J}[m] &\rightarrow &\mu_m\\
\end{eqnarray*}
is a bilinear, non-degenerate map and it commutes with the action of $G$. If $\lambda: A \rightarrow \hat{A}$ is a polarization,
then we define the Weil pairing as
\begin{eqnarray*}
W_m^{\lambda}:J[m]\times J[m] &\rightarrow &\mu_m\\
(P,Q)& \rightarrow & W_m(P,\lambda(Q)).
\end{eqnarray*}
Given a subgroup $G\subset J[m]$, we say that $G$ is \textit{isotropic} with respect to the Weil pairing if the Weil pairing restricted to $G\times G$ is trivial. It is \textit{maximal isotropic} if it is isotropic and it is not properly contained in any other such subgroup.
\noindent
We denote by $H^i(G_{\bar{\F}_q/\F_q},J)$ the $i$-th Galois cohomology group, for $i\geq 0$. 

Consider the exact sequence $0\rightarrow J[m]\rightarrow J(\bar{\F}_q)\rightarrow J(\bar{\F_q})\rightarrow 0$. Then by taking Galois cohomology we get the connecting morphism
\begin{eqnarray*}
\delta: J(\F_q)/mJ(\F_q)=H^0(G_{\bar{\F}_q/\F_q},J)/mH^0(G_{\bar{\F}_q/\F_q},J)&\rightarrow & H^1(G_{\bar{\F}_q/\F_q},J[m])\\
P&\rightarrow & F_P,
\end{eqnarray*}
where the map $F_P$ is defined as follows
\begin{eqnarray*}
F_P:G_{\bar{\F}_q/\F_q} & \rightarrow & J(\bar{\F}_q)[m]\\
\sigma & \rightarrow & \sigma(\bar{P})-\bar{P},
\end{eqnarray*}
where $\bar{P}$ is any point such that $m\bar{P}=P$.
Using the connecting morphism and the Weil pairing, we define the $m$-Tate pairing as follows
\begin{eqnarray*}
t_m:J(\F_q)/mJ(\F_q)\times \hat{J}[m](\F_q)&\rightarrow &H^1(G,\mu_m)\\
(P,Q) &\rightarrow & [\sigma \rightarrow W_m(F_P(\sigma),Q)].
\end{eqnarray*}
For a fixed polarization $\lambda:J\rightarrow \hat{J}$ we define a pairing on $J$ itself
\begin{eqnarray*}
t_m^{\lambda}:J(\F_q)/mJ(\F_q)\times J[m](\F_q)&\rightarrow &\F_q^*/\F_q^{*m}\\
(P,Q) &\rightarrow & t_m(P,\lambda(Q)).
\end{eqnarray*}
Most often, if $J$ has a distinguished principal polarization and there is no risk of confusion, we write simply $t_m(\cdot,\cdot)$ instead of $t_m^{\lambda}(\cdot,\cdot)$.

Lichtenbaum~\cite{Lichtenbaum} describes a version of the Tate pairing on Jacobian varieties. More precisely, suppose we have $m|\#J(\F_q)$ and denote by $k$ the \textit{embedding degree with respect to } $m$, i.e. the smallest integer $k\geq 0$ such that $m|q^k-1$. Let $D_1\in J(\F_{q^k})$ and $D_2\in J[m](\F_{q^k})$ two divisor classes, represented by two divisors such that $\textrm{supp}(D_1)\cap \textrm{supp}(D_2)=\emptyset$. Since $D_2$ has order $m$, there is a function $f_{m,D_2}$ such that $\textrm{div}(f_{m,D_2})=mD_2$. The Tate pairing of the divisor classes $D_1 $ and $D_2$ is computed as
\begin{eqnarray*}
t_m(D_1,D_2)=f_{D_2}(D_1).
\end{eqnarray*}
Moreover, in computational applications, it is convenient to work with a unique value of the pairing. Given that $\F_{q^k}^*/(\F_{q^k}^*)^m\simeq \mu_m$, we use the \textit{reduced Tate pairing}, given by
\begin{eqnarray*}
T_m(\cdot,\cdot): J(\F_{q^k})/mJ(\F_{q^k})\times J[m](\F_{q^k})&\rightarrow &\mu_m\\
(P,Q) &\rightarrow & t_m(P,Q)^{(q^k-1)/m}.
\end{eqnarray*}
The function $f_{m,D_2}(D_1)$ is computed using Miller's algorithm~\cite{Miller} in $O(\log m)$ operations in $\F_{q^k}$.
Since $H^1(G_{\bar{\F}_{q^k}/\F_{q^k}},\mu_m)\simeq \mu_m$ by Hilbert's 90 theorem, it follows that there is an isomorphism $H^1(G_{\bar{\F}_{q^k}/\F_{q^k}},\mu_m)\simeq H^1(\textrm{Gal}(\F_{q^{km}}/\F_{q^k}),\mu_m)$. Since $H^1(Gal(\F_{q^{km}}/\F_{q^k}),\mu_m)\simeq \mu_m$, we may compute the Tate pairing as
\begin{eqnarray*}
t_m(\cdot,\cdot):J(\F_{q^k})/mJ(\F_{q^k})\times \hat{J}[m](\F_{q^k})&\rightarrow &\mu_m\\
(P,Q) &\rightarrow &  W_m(F_P(\pi),Q),
\end{eqnarray*}
where $\pi$ is the Frobenius of the finite field $\F_{q^k}$.
\section{Pairings and endomorphism ring computation}\label{MaximalOrders}
In this section we relate some properties of the Tate pairing to the isomorphism class of the endomorphism ring of the Jacobian. Let $\ell$ be a prime odd number. We give a method to check whether the endomorphism ring is locally maximal at $\ell$ (i.e. the index $[\mathcal{O}_K:\mathcal{O}]$ is not divisible by $\ell$) by computing a certain number of pairings.

Let $H$ be a genus 2 smooth irreducible curve defined over a finite field
$\F_q$, $J$ its jacobian and suppose that $J[\ell^n]\subseteq J(\F_q)$ and
that $J[\ell^{n+1}]\nsubseteq J(\F_q)$.

\begin{lemma}\label{antisymmetry}
The reduced Tate pairing defined as
\begin{eqnarray*}
T_{\ell^n}:J[\ell^n]\times J[\ell^n]\rightarrow \mu_{\ell^n}
\end{eqnarray*}
is $k_{\ell}$-antisymmetric, i.e. $T_{\ell^n}(\bar{D}_1,\bar{D}_2)T_{\ell^n}(\bar{D}_2,\bar{D}_1)\in \mu_{\ell^{k_{\ell}}}$, for all $\bar{D}_1,\bar{D}_2\in J[\ell^n]$.
\end{lemma}
\begin{proof}
Indeed, assume that there are $\bar{D}_1,\bar{D}_2\in J[\ell^n]$ such that $T_{\ell^n}(\bar{D}_1,\bar{D}_2)T_{\ell^n}(D_2,\bar{D}_1)\in \mu_{\ell^n}\backslash \mu_{\ell^{k_{\ell}}}$.
We denote by $G=\langle \bar{D}_1,\bar{D}_2 \rangle$ and by $r>k_{\ell}$ the largest integer such that $T_{\ell^n}(\bar{D}_1,\bar{D}_2)T_{\ell^n}(\bar{D}_2,\bar{D}_1)$ is an $\ell^r$-th primitive root of unity. Then the polynomial
\begin{eqnarray*}
\mathcal{P}(a,b)=\log T_{\ell^n}(\bar{D}_1,\bar{D}_1)a^2+\log (T_{\ell^n}(\bar{D}_1,\bar{D}_2)T_{\ell^n}(\bar{D}_2,\bar{D}_1))ab+\log T_{\ell^n}(\bar{D}_2,\bar{D}_2)b^2,
\end{eqnarray*}
where the $\log$ function is computed with respect to some fixed $\ell^n$-th root of unity, is zero $\bmod~\ell^{n-r-1}$ and non-zero $\bmod~\ell^{n-r}$. Dividing by $\ell^{n-r-1}$, we may view $\mathcal{P}$ as a polynomial in $\F_{\ell}[a,b]$. Since $\mathcal{P}$ is a quadratic non-zero polynomial, it has at most two roots. These correspond to two divisor classes in $G$, with $r$-degenerate self-pairing. Hence, there is at least one divisor $\bar{D}\in G$ such that $T_{\ell^n}(\bar{D},\bar{D})$ is a $\ell^{r}$-th root of unity. Since there is at least one maximal isotropic subgroup $W\in \mathcal{W}$ with respect to the Weil pairing such that $\bar{D}\in W$, this contradicts the definition of $k_{\ell}$.
\end{proof}
Let $\mathcal{O}$ be an order of $K$ and let $\theta\in \mathcal{O}$. We define
$$v_{\ell,\mathcal{O}}(\theta):=\max _{m\geq 0}\{m:\theta \in \Z+\ell^m\mathcal{O}\}.$$
We denote by $1,\delta,\gamma,\eta$ a $\Z$-basis of $\mathcal{O}$ and and we write $\theta=a_1+a_2\delta+a_3\gamma+a_4\eta$. Then we compute $v_{\ell,\mathcal{O}}$ as
\begin{eqnarray}\label{basis}
v_{\ell,\mathcal{O}}(\theta)=v_{\ell}(\gcd(a_2,a_3,a_4)).
\end{eqnarray}
Note that the value of $v_{\ell,\mathcal{O}}(\theta)$ is independent of the choice of the basis.
We say that $\theta$ is divisible by $t\in \Z$ if we have $\theta\in t\mathcal{O}$. We say that $\theta$ is exactly divisible by $\ell^n$ if it is divisible by $\ell^n$ and it is not divisible by $\ell^{n+1}$.
The following lemma gives a criterion to check whether an order is locally maximal at $\ell$ or not.

\begin{lemma}\label{CharacterizationMaximalOrder}
Let $K:=\Q(\eta)$ be a quartic CM field, with $\eta=i\sqrt{a+b\frac{-1+\sqrt{d}}{2}}$, if $d\equiv 1 \bmod 4$ and $\eta=i\sqrt{a+b\sqrt{d}}$, if $d\equiv 2,3 \bmod 4$. We assume that $a,b,d\in \Z$ and that $d$ and $a^2-b^2d$ are square free. Assume that $K_0=\Q(\sqrt{d})$ has class number 1. Let $\ell>2$ a prime number that does not divide $\textrm{lcm}(a,b,d)$. Let $\mathcal{O}_K$ be the maximal order of $K$ and $\mathcal{O}$ an order such that $[\mathcal{O}_K: \mathcal{O}]$ is divisible by $\ell$.
Let $\pi\in \mathcal{O}$ such that $N_{K/K_0}(\pi)\in \Z$ is not divisible by $\ell$ and that $v_{\ell,\mathcal{O}_K}(\pi)>0$. We suppose that $\pi=a_1+a_2\frac{-1+\sqrt{d}}{2}+(a_3+a_4\frac{-1+\sqrt{d}}{2})\eta$, if $d\equiv 1 \bmod 4$ and $\pi=a_1+a_2\sqrt{d}+(a_3+a_4\sqrt{d})\eta$, if $d\equiv 2,3 \bmod 4$. If $\max (v_{\ell}(\frac{a_3-a_4}{\ell}),v_{\ell}(\frac{a_3-\ell a_4}{\ell^2})) <\min (v_{\ell}(a_3), v_{\ell}(a_4))$, then $v_{\ell,\mathcal{O}}(\pi)<v_{\ell,\mathcal{O}_K}(\pi)$.
\end{lemma}

\begin{proof}
We denote by $\mathcal{O}_1=\mathcal{O}_{K_0}+\mathcal{O}_{K_0}\eta$. Since $\ell>2$, it suffices to show that $v_{\ell,\mathcal{O}\cap \mathcal{O}_1}(\pi)<v_{\ell,\mathcal{O}_1}(\pi)$. We will therefore assume, without restricting the generality, that $\mathcal{O}\subset \mathcal{O}_1$.
Let $\delta=\frac{-1+\sqrt{d}}{2}$ if $d\equiv 1 \bmod 4$ and $\delta=\sqrt{d}$, if $d\equiv 2,3 \bmod 4$ and let $\gamma:=\delta\eta$.
Then $1,\delta,\gamma,\eta$ is a basis for $\mathcal{O}_1$. We write $\pi=a_1+a_2\delta+a_3\gamma+a_4\eta$. By writing down the norm condition for $d\equiv 2,3 \bmod 4$ 
\begin{eqnarray*}
\left (a_1+a_2\sqrt{d}+(a_3+a_4\sqrt{d})i\sqrt{a+b\sqrt{d}}\right )\left (a_1+a_2\sqrt{d}-(a_3+a_4\sqrt{d})i\sqrt{a+b\sqrt{d}}\right ) \in \Z,
\end{eqnarray*}
we get that
\begin{eqnarray}\label{NormEquation1}
2a_1a_2+a_3^2b+a_4^2bd+2aa_3a_4=0.
\end{eqnarray}
Similarly, for $d\equiv 1 \bmod 4$, we have
\label{NormEquation2}
\begin{eqnarray}
  -\frac{a_2^2}{2}+a_1a_2-\frac{aa_4^2}{2}+a_3a_4(a-b)+\frac{a_3^2b}{2}+\frac{a_4^2(1+d)b}{8}-\frac{a_4^2(2a-b)}{4}=0.
\end{eqnarray}
Since $\ell\nmid a_1$, equations~\eqref{NormEquation1} and~\eqref{NormEquation2} imply that $v_{\ell}(a_2)>\textrm{min}(v_{\ell}(a_3),v_{\ell}(a_4))$. Since there is always an order $\mathcal{O}'$ such that $\mathcal{O}\subset \mathcal{O}'\subset\mathcal{O}_1$ such that $[\mathcal{O}_1:\mathcal{O}']$ is a power of $\ell$, it suffices to prove the lemma in the case $[\mathcal{O}_1:\mathcal{O}]$ is a power of $\ell$. For the order $\mathcal{O}$, we choose $\{1,\delta',\gamma',\eta'\}$ a HNF basis with respect to $\{1,\delta,\gamma,\eta\}$.
We denote by $(a_{i,j})_{1\leq i,j\leq 4}$ the corresponding transformation matrix. Then $[\mathcal{O}_1:\mathcal{O}]=\prod _{1\leq i\leq 4} a_{i,i}$. Note that neither $\eta$ nor $\gamma$ are in $\mathcal{O}$. Otherwise, $\mathcal{O}$ is the maximal order. Indeed, assume $\eta\in \mathcal{O}$. Since $\ell$ divides neither $a$ nor $b$, it follows that $\delta\in \mathcal{O}$. This implies that $\mathcal{O}$ is $\mathcal{O}_1$.
We consider the decomposition of $\pi$ over the basis $\{1,\delta',\gamma',\eta'\}$
\begin{eqnarray*}
\pi=a'_1+a'_2\delta'+a'_3\gamma'+a'_4\eta',a'_i\in \Z.
\end{eqnarray*}
Since $\eta\notin \mathcal{O}$, we know that $a_{44}$ is $\ell$. If $a_{33}$ is divisible by $\ell$, then $v_{\ell}(a'_3)<v_{\ell}(a_3)$. If $a_{34}=1$, then $a'_4=-(a_3-\ell a_4)/\ell^2$. If $a_{34}=0$, then $a'_4=a_4/\ell$. If $a_{33}=1$, it follows that $a_{34}=1$ (otherwise we would have $\gamma\in \mathcal{O}$). Then $a'_3=a_3$ and $a'_4=-(a_3-a_4)/\ell$. We conclude that $v_{\ell,\mathcal{O}}(\pi)<v_{\ell,\mathcal{O}_K}(\pi)$.
\end{proof}

Since we know that $J[\ell^n]$ is $\F_q$-rational, while $J[\ell^{n+1}]$ is
not, Lemma~\ref{torsion} implies that $\pi-1$ is exactly divisible by
$\ell^n$. Moreover, the Frobenius matrix on the Tate module is the identity
matrix $I_4\bmod \ell^n$. In following lemma, we compute the matrix of the Frobenius on the Tate module.

\begin{lemma}\label{Matrix}
Let $J$ be a abelian surface defined over a finite field $\F_q$ and $\pi$ the Frobenius endomorphism. Then the largest integer $m$ such that the matrix of the Frobenius endomorphism on the $\ell$-Tate module is of the form
\begin{eqnarray}\label{FrobeniusMatrix}
\left (\begin{array}{cccc}
\lambda & 0 & 0 & 0 \\
0 & \lambda & 0 & 0\\
0 & 0 & \lambda & 0\\
0 & 0 & 0 & \lambda \\
\end{array}
\right )\bmod \ell^m
\end{eqnarray}
is $v_{\ell,\mathcal{O}}(\pi)$, where $\mathcal{O}$ is the endomorphism ring of $J$.
\end{lemma}
\begin{proof}
Let $m$ be the largest integer such that the matrix of the Frobenius on $J[\ell^{m}]$ has the form given in Equation~\eqref{FrobeniusMatrix}. Let $\mathcal{O}$ be the endomorphism ring of $J$. We denote by $\{1,\delta,\gamma,\eta\}$ the $\Z$-basis of $\mathcal{O}$ and by $\pi=a_1+a_2\delta+a_3\gamma+a_4\eta$ the decomposition of $\pi$ over this basis.
It is obvious that $ m \geq v_{\ell}(\textrm{gcd}(a_2,a_3,a_4))$. For the converse, we note that $\pi-\lambda$ kills the $\ell^m$-torsion, hence we may write $\pi-\lambda=\ell^m\alpha$, with $\alpha \in \textrm{End}(J)$. We write down the decomposition of $\alpha$ over the basis $\{1, \delta,\gamma,\eta \}$ and conclude that  $\ell^m|\gcd{(a_2,a_3,a_4)}$. Hence $m\leq v_{\ell}(\textrm{gcd}(a_2,a_3,a_4)).$ We conclude that $m=v_{\ell}(\textrm{gcd}(a_2,a_3,a_4)),$ hence $m=v_{\ell,\mathcal{O}}(\pi)$ by~\eqref{basis}.
\end{proof}
Using Galois cohomology, Schmoyer~\cite{Schmoyer} computes the matrix of the Frobenius on the Tate module, up to a certain precision, if the self-pairings of the Tate pairing are degenerate. We use a similar approach and show that the precision up to which the Frobenius acts on the Tate module as a multiple of the identity is $2n-k_{\ell}$. Consequently, we recover information on the conductor of the endomorphism ring of $J$ by computing $k_{\ell}$. For $m\in \Z$, we will use a \textit{symplectic basis} of $J[\ell^m]$, i.e. a basis such that the matrix associated to the $\ell^m$-Weil pairing is
\begin{eqnarray}\label{symplectic}
\left (\begin{array}{cc}
0 & I \\
-I & 0\\
\end{array}
\right )\bmod \ell^m.
\end{eqnarray}

\begin{proposition}\label{PropPrinc}
Let $H$ be a hyperelliptic smooth irreducible curve defined over a finite field $\F_q$, and $J$ its jacobian. Suppose that the Frobenius endomorphism $\pi$ is such that $\pi-1$ is exactly divisible by $\ell^n$, for $\ell\geq 3$ prime. Then if $v_{\ell,\textrm{End}(J)}(\pi)<2n$, we have
\begin{eqnarray}
v_{\ell,\textrm{End}(J)}(\pi)=2n-k_{\ell}.
\end{eqnarray}
\end{proposition}
\begin{proof}
Let $\{Q_{1},Q_2,Q_{-1},Q_{-2}\}$ a symplectic basis for the $\ell^{2n}$-torsion (whose matrix is given by Equation~\eqref{symplectic}) and let $\pi(Q_g)=\sum_{h=-2}^{2}a_{h,g}Q_h$, with $(a_{h,g})_{h,g\in \{-2,-1,1,2\}}$ in $\Z$.
By bilinearity, we have that
\begin{eqnarray*}
T_{\ell^n}(\ell^n Q_i,\ell^n Q_j)&=&W_{\ell^{2n}}(Q_i,\pi(Q_j)-Q_j)=W_{\ell^{2n}}(Q_i,\sum _{\substack{h=-2\\h\neq 0}}^{2}a_{h,j}Q_h-Q_j)\\&&=W_{\ell^{2n}}(Q_i,Q_j)^{a_{j,j}-1}\prod _{\substack{h=-2\\ h \neq 0,j}}^2 W_{\ell^{2n}}(Q_i,Q_h)^{a_{h,j}}.
\end{eqnarray*}
If $j \neq -i$, we have that $T_{\ell^n}(\ell^n Q_i,\ell^n Q_j)\in \mu_{\ell^{k_{\ell}}}$. It follows that
\begin{eqnarray}\label{jale}
a_{-i,j}\equiv 0 \pmod {\ell^{2n-k_{\ell}}},
\end{eqnarray}
for $i\neq -j$. If $j=-i$, then $T_{\ell^n}(\ell^n Q_i,\ell^n Q_j)=W_{\ell^{2n}}(Q_i,Q_j)^{a_{j,j}-1}$. Since the Tate pairing is $k_{\ell}$-antisymmetric we get
\begin{eqnarray*}
a_{i,i}&\equiv &a_{-i,-i}\pmod {\ell^{2n-k_{\ell}}}.
\end{eqnarray*}
It remains to prove that $a_{i,i}\equiv a_{j,j}$ , for $i,j\in \{-2,-1,1,2\}$.
Note that by Galois invariance, we have $W_{\ell^{2n}}(\pi(Q_i),\pi(Q_j))=\pi(W_{\ell^{2n}}(Q_i,Q_j))=W_{\ell^{2n}}(Q_i,Q_j)^q$.
For $i=-j$ we have
\begin{eqnarray*}
&&W_{\ell^{2n}}(\pi(Q_i),\pi(Q_{-i}))=W_{\ell^{2n}}(\sum_{\substack{h=-2\\h\neq 0}}^{2} a_{h,i}Q_h,\sum _{\substack{g=-2\\g\neq 0}}^2 a_{g,-i}Q_{g})\\&=& \prod _{\substack{h=-2\\h\neq 0}}^2 \prod_{\substack{g=-2\\g\neq 0}}^2W_{\ell^{2n}}(a_{h,i}Q_h,a_{g,-i}Q_{g})=W_{\ell^{2n}}(Q_i,Q_{-i})^{a_{i,i}a_{-i,-i}}\prod _{\substack{h=-2\\h\neq 0,i}}^2 W_{\ell^{2n}}(a_{h,i}Q_h, a_{-i,-i}Q_{-i})\\&& \cdot \prod _{\substack{g=-2\\g\neq 0,-i}}^2 W_{\ell^{2n}}(a_{i,i}Q_i, a_{g,-i}Q_{g})\prod _{\substack{s=-2\\s\neq 0,i}}^2\prod _{\substack{t=-2\\t\neq 0,-i}}^2W_{\ell^{2n}}(Q_s,Q_{t})^{a_{s,i}a_{t,-i}}
\end{eqnarray*}
Since $\{Q_{1},Q_{2},Q_{-1},Q_{-2}\}$ is a symplectic basis and that $a_{h,g}\equiv 0 \pmod {\ell^n}$, for $h\neq -g$, then
\begin{eqnarray*}
W_{\ell^{2n}}(\pi(Q_i),\pi(Q_{-i}))=W_{\ell^{2n}}^{a_{i,i}a_{-i,-i}}(Q_i,Q_{-i}).
\end{eqnarray*}
Since $a_{i,i}\equiv a_{-i,-i} \pmod {\ell^{2n-k_{\ell}}}$, it follows that
$$a_{i,i}^2\equiv q~\textrm{for all}~i\in \{-2,-1,1,2\}.$$
Since $a_{i,i}\equiv 1 \pmod {\ell^n}$, it follows that $a_{i,i}\equiv b \pmod {\ell^{2n-k_{\ell}}}$, for some $b\in \Z$. By Lemma~\ref{Matrix},
we have $2n-k_{\ell}\leq v_{\ell,\textrm{End}J}(\pi)$.
For the converse, let $k=2n-v_{\ell,\textrm{End}J}(\pi)$ and $R,S$ be two points in $J[\ell^n]$ such that $W_{\ell}(R,S)=1$. It suffices to show that $T_{\ell^n}(R,S)$ is $k$-degenerate.
We write $\pi-1=a_1+a_2\alpha+a_3\beta+a_4\theta$, where $1,\alpha,\beta,\theta$ form a $\Z$-basis of $\textrm{End}(J)$. We take $\bar{S}$ such that $S=\ell^n\bar{S}$ and we get
\begin{eqnarray*}
&&T_{\ell^n}(R,S)=W_{\ell^{n}}(R,(\pi-1)(\bar{S}))=\\&&= W_{\ell^n}(R,S)^{\frac{a_1}{\ell^n}}W_{\ell^{n}}(R,(\frac{a_2}{\ell^{2n-k}}\delta+\frac{a_3}{\ell^{2n-k}}\gamma+\frac{a_4}{\ell^{2n-k}}\eta)(S))^{\ell^{n-k}}.
\end{eqnarray*}
Since $W_{\ell}(R,S)=1$ and $v_{\ell}(\gcd(a_2,a_3,a_4))=\ell^{2n-k}$, we have $T_{\ell^n}(R,S)\in \mu_{\ell^k}$. Hence $k\geq k_{\ell}$. This concludes the proof.
\end{proof}
\noindent
Proposition~\ref{PropPrinc} gives a method to compute to compute $v_{\ell,\textrm{End}J}(\pi)$ using pairings. Together with Lemma~\ref{CharacterizationMaximalOrder}, this gives a criterion to check whether the endomorphism ring of a jacobian is locally maximal at $\ell$.  

\begin{theorem}\label{MainResult}
Let $H$ be a smooth irreducible genus 2 curve defined over a finite field $\F_q$ and $J$ its jacobian. Suppose that the Frobenius endomorphism $\pi$ is exactly divisible by $\ell^n$, $n\in \Z$ and that the conditions in Lemma~\ref{CharacterizationMaximalOrder} are satisfied. Then if $v_{\ell,\mathcal{O}_K}(\pi)<2n$, $\textrm{End}(J)$ is a locally maximal order at $\ell$ if and only if $k_{\ell}$ equals $2n-v_{\ell,\mathcal{O}_K}(\pi)$.
\end{theorem}
\begin{proof}
By Proposition~\ref{PropPrinc}, $k_{\ell}$ equals $2n-v_{\ell^n,\mathcal{O}}(\pi)$, where $\mathcal{O}\simeq \textrm{End}(J)$. By Lemma~\ref{CharacterizationMaximalOrder}, the value of $v_{\ell^n,\mathcal{O}_K}(\pi)$ uniquely characterizes orders which are locally maximal at $\ell$.
\end{proof}

\begin{remark}
Let $\pi=1+a_1+a_2\delta+a_3\gamma+a_4\eta$ be the decomposition of the
Frobenius over a $\Z$-basis of $\mathcal{O}_K$. We deduce that $k_{\ell}>0$
if and only if $v_{\ell}(\gcd(a_2,a_3,a_4))<2v_{\ell}(\gcd{(a_1,a_2,a_3,a_4}))$.
\end{remark}

\noindent
We conclude this section by giving in Algorithm~\ref{LocallyMaximal} a computational method which verifies whether the jacobian $J$ of a genus 2 curve has locally maximal endomorphism ring. If $k_{\ell}=0$, the algorithm aborts. By Lemma~\ref{Matrix}, computing $k_{\ell}$ is equivalent to computing the greatest power of $\ell$ dividing all coefficients $a_{i,j}$, with $i\neq j$ of the matrix of the Frobenius on the Tate module. Equation~\ref{jale} shows that in order to compute the $\ell$-adic valuation of these coefficients, it suffices to determine all the values $T_{\ell^n}(Q_i,Q_j)$, for $i\neq j$.

\begin{algorithm}[h!]
\caption{Checking whether the endomorphism ring is locally maximal}\label{LocallyMaximal}
\begin{algorithmic}[1]
\REQUIRE A jacobian $J$ of a genus 2 curve defined over $\F_q$ such that
$J[\ell^n]\subset J(\F_q)$, the Frobenius $\pi$, a symplectic basis $(Q_1,Q_{2},Q_{-1},Q_{-2})$ for $J[\ell^n]$
\ENSURE The algorithm outputs true if $\textrm{End}(J)$ is maximal at $\ell$.
\FORALL{$i,j \in \{1,2,-1,-2\}$}
\IF {$i\neq -j$}
\STATE Compute $t_{i,j}\leftarrow T_{\ell^{n}}(Q_i,Q_j),$
\ELSE \STATE $t_{i,j}\leftarrow T_{\ell^n}(Q_i,Q_j)T_{\ell^n}(Q_j,Q_i)$
\ENDIF
\ENDFOR
\STATE Let $\mbox{Count}\leftarrow 0$ and $\mbox{check}\leftarrow -1$. 
\WHILE {$\mbox{check}\neq \mbox{Count}$}
\STATE $\mbox{check}\leftarrow \mbox{Count}$
\FORALL {$i,j \in \{1,2,-1,-2\}$}
\IF {$t_{i,j}\neq 1$}
\STATE Let $t_{i,j}=t_{i,j}^\ell$
\STATE $\mbox{check}\leftarrow -1$
\ENDIF
\ENDFOR
\IF {$\mbox{check}\neq \mbox{Count}$}
\STATE $\mbox{Count}=\mbox{Count}+1$
\ENDIF
\ENDWHILE
\STATE $k_{\ell}\leftarrow n-\mbox{Count}$
\IF {$\mbox{Count}=0$} \STATE \textbf{abort}
\ENDIF
\IF {$k_{\ell}=2n-v_{\ell,\mathcal{O}_K}(\pi)$} \STATE \textbf{return} true
\ELSE \STATE \textbf{return} false
\ENDIF
\end{algorithmic}
\end{algorithm}
\section{Application to horizontal isogeny computation}\label{HorizontalIsogenies}

In this section, we are interested in computing \textit{horizontal} isogenies, i.e. isogenies between Jacobians having the same endomorphism ring. Note that if $I:J_1\rightarrow J_2$ is an isogeny such that $J_1$ has maximal endomorphism ring at $\ell$, we distinguish two cases: either $\textrm{End}(J_2)$ is locally maximal at $\ell$, or $\textrm{End}(J_2)\subset \textrm{End}(J_1)$. In the last case we say that the isogeny is \textit{descending}.

Over the complex numbers, horizontal isogenies are given in terms of the action of the Shimura class group~\cite{Shimura}.
Let $\Phi$ be a CM-type and let $A$ be an abelian surface over $\C$ with complex multiplication by $\mathcal{O}_K$, given by $A=\C^2/\Phi(I^{-1})$, where $I$ is an ideal of $\mathcal{O}_K$. The surface is principally polarized if there is a purely imaginary $\xi \in \mathcal{O}_K$ with $\textrm{Im}(\Phi_i(\xi))>0$, for $i\in \{1,2\}$, and such that $\xi \mathfrak{D}_K=I\bar{I}$ (where $\mathfrak{D}_K$ is the different $\{\alpha \in \mathcal{O}_K: \mathrm{Tr}_{K/\Q}(\alpha \mathcal{O}_K)\subset \Z\}$).
Computing horizontal isogenies is usually done by using the action of the Shimura class group~\cite{Shimura}.
This group, that we denote by $\textgoth{C}(K)$, is defined as
$$\{(\textfrak{a},\alpha)|\mathfrak{a}~\textrm{is a fractional}~\mathcal{O}_K\mbox{-}\textrm{ideal with}~\mathfrak{a}\mathfrak{\bar{a}}=(\alpha)~\textrm{with}~\alpha\in K_0~\textrm{totally positive}\}/\sim,$$ where $(\mathfrak{a},\alpha)\sim (\mathfrak{b},\beta)$ if and only if there exists $u\in K^{*}$ with $\mathfrak{b}=u\mathfrak{a}$ and $\beta=u\bar{u}\alpha$.
The action of $(\mathfrak{a},\alpha)\in \textgoth{C}(K)$ on an principally polarized abelian surface given by $(I,\xi)$ is given by the ideal
$(\mathfrak{a}I,\alpha\xi)$. This action is transitive and free~\cite[\textsection 14.6]{Shimura}.

If the norm of $\mathfrak{a}$ is coprime to the discriminant of
$\Z[\pi,\bar{\pi}]$, the kernel of the horizontal isogeny corresponding to
$\mathfrak{a}$ is a subgroup of the $\ell$-torsion invariant under the
Frobenius endomorphism. Hence in order to compute the kernel, we need to
compute the matrix of the Frobenius for some basis of the $\ell$-torsion
and then determine subspaces which are invariant by this matrix
(see~\cite[Algorithm VI.3.4]{Bisson}). We show that, when a Jacobian with
locally maximal order at $\ell$ is given, kernels of
$(\ell,\ell)$-horizontal isogenies are subgroups on which the Tate pairing
is degenerate. This result holds for any $\ell>2$ and is independent of the
value of the discriminant of $\Z[\pi,\bar{\pi}]$. The resulting algorithm,
whose complexity is analysed in Section~\ref{Complexity}, computes kernels
of horizontal isogenies with only a few pairing computations.

We state the following lemma for jacobians of genus 2 curves over finite
fields, which are the framework for this paper. We note that the result
holds for abelian varieties in general. 

\begin{lemma}\label{IsogenyPairing}
\begin{itemize}
\item[(a)] Let $J_1,J_2$ be jacobians of genus 2 smooth irreducible curves defined over a finite field $\F_q$ and $I:J_1\rightarrow J_2$ an isogeny defined over $\F_q$ which splits multiplication by $d$. Let $\lambda: J_1\rightarrow \hat{J}_1$ be a principal polarization. Then for $P\in J_1(K)$, $Q\in J_1[m](K)$ we have
\begin{eqnarray*}
T_m^{\lambda_I}(I(P),I(Q))=T_m^{\lambda}(P,Q)^d,
\end{eqnarray*}
where $\lambda_I: J_2\rightarrow \hat{J_2}$ is the principal polarization such that $I\circ \lambda_I \circ \check{I}=d\circ \lambda$.
\item[(b)]  Let $J_1,J_2$ be jacobians of genus 2 smooth irreducible curves defined over $\F_q$ and $I:J_1\rightarrow J_2$ an isogeny defined over $\F_q$ which splits multiplication by $m$. Let $P\in J_1(K)$, $Q\in J_1[mm'](K)$ such that $I(Q)$ is a $m'$-torsion point.
\begin{eqnarray*}
T_{m'}^{\lambda_I}(I(P),I(Q))=T^{\lambda}_{mm'}(P,Q)^m,
\end{eqnarray*}
where $\lambda_I$ is a principal polarization of $J_2$ such that $I\circ \lambda_I\circ \check{I}=m\circ \lambda$.
\end{itemize}
\end{lemma}
\begin{proof}
(a) It is easy to check that $\delta(I(P))=I(\delta(P))$. Hence for $\sigma \in G_K$ we have
$$W_m(F_{I(P)}(\sigma),I(Q))=W_m(I(F_P(\sigma)),I(Q)).$$
By using~\cite[Proposition 13.2.b]{Milne}
$$W_m^{\lambda_I}(I(F_P(\sigma)),I(Q))=W_m^{\check{I}\circ \lambda_I\circ I}(F_P(\sigma),Q).$$
(b) The proof is immediate by using (a) and the fact that $T_{mm'}(I(P),I(Q))=T_{m'}(I(P),I(Q))$.
\end{proof}

\begin{lemma}\label{divisors}
Let $H/\F_q$ be a smooth irreducible curve and $D_1,D_2$ are two elements of $J(\F_q)$ of order
$\ell^n$, $n\geq 1$. Let $\bar{D}_1,\bar{D}_2\in J(\F_q)$ such that $\ell \bar{D}_1=D_1$ and $\ell \bar{D}_2=D_2$. Then we have
\begin{itemize}
\item[(a)] If $\bar{D}_1,\bar{D}_2\in J(\F_q)$, then
$$T_{\ell^{n+1}}(\bar{D}_1,\bar{D}_2)^{\ell^2}=T_{\ell^n}(D_1,D_2).$$
\item[(b)] Suppose $\ell\geq 3$. If $\bar{D}_1\in J(\bar{\F}_q)\backslash J(\F_q)$, then
$$T_{\ell^{n+1}}(\bar{D}_1,\bar{D}_2)^{\ell}=T_{\ell^n}(D_1,D_2).$$
\end{itemize}
\end{lemma}
\begin{proof}
The proof is similar to to the one of~\cite[Lemma 4.6]{IonJou}. For completeness, we detail it in Appendice~\ref{Annexe1}. 
\end{proof}

\noindent
We may now prove Theorem~\ref{horizontal}.

\vspace{0.2 cm}

\noindent
\textit{Proof of Theorem 1.}  We assume that $k_{\ell}\geq 2$. Otherwise,
we use Lemma~\ref{divisors}  and work over an extension field of $\F_q$. We
denote by $I:J\rightarrow  J'$ the isogeny of kernel $G$. Let $k'_{\ell}$ be the $k_{\ell}$ corresponding to $J'$.\\
1) Suppose that $\bar{G}$ is such that the Tate pairing is non-degenerate over $\bar{G}\times \bar{G}$. Then by applying Lemma~\ref{IsogenyPairing} we have
$$T_{\ell^{n-1}}(I(P_1),I(P_2))\in \mu_{\ell^{k_{\ell}-1}}\backslash \mu_{\ell^{k_{\ell}-2}},$$
for $P_1,P_2\in \bar{G}$. If $J'[\ell^n]$ is not defined over $\F_q$, then its endomorphism ring cannot be maximal at $\ell$, hence the isogeny is descending. Assume then that $J'[\ell^n]$ is defined over $\F_q$.
Let $\bar{P}_1,\bar{P}_2\in J'[\ell^n]$ be such that $I(P_1)=\ell \bar{P}_1$, $I(P_2)=\ell \bar{P}_2$. Then $T_{\ell^n}(\bar{P}_1,\bar{P}_2)\in \mu_{\ell^{k_{\ell}+1}}\backslash \mu_{\ell^{k_{\ell}}}$.
We denote by $G'=<\bar{P}_1,\bar{P}_2>$. The subgroup $G'$ may be chosen such that it is maximal isotropic with respect to the $\ell^n$-Weil pairing. It follows that $k_{\ell}'\geq k_{\ell}+1$. By Theorem~\ref{MainResult}, we deduce that the endomorphism ring of $J'$ is not locally maximal at $\ell$, hence the isogeny is descending.\\
2) Suppose now that the Tate pairing is degenerate over $\bar{G}\times \bar{G}$. We distinguish two cases.\\
\textit{Case 1.} Suppose that $J'[\ell^n]$ is defined over $\F_q$. With the same notations as above, we get that $T_{\ell^n}(\bar{P}_1,\bar{P}_2)\in \mu_{\ell^{k_{\ell}}}$. Let $L\subset J'[\ell^n]$ be a subgroup of rank 2 maximal isotropic with respect to the Weil pairing and consider $Q_1,Q_2\in L\backslash G'$. Then $\ell^{n-1} Q_1,\ell^{n-1}Q_2\in Ker~I^{\dagger}$. Since $T_{\ell^{n-1}}(I^{\dagger}(Q_1),I^{\dagger}(Q_2))\in \mu_{\ell^{k_{\ell}-2}}$, it follows that $T_{\ell^n}(Q_1,Q_2)\in \mu_{\ell^{k_{\ell}-1}}$. Hence $k_{\ell}'\leq k_{\ell}$. By Theorem~\ref{MainResult}, we conclude that the endomorphism ring of $J'$ is locally maximal at $\ell$. \\
\textit{Case 2.} Suppose that $J'[\ell^n]$ is not defined over $\F_q$. Hence $I$ is descending. We have
$$T_{\ell^{n-1}}(I(P_1),I(P_2))\in \mu_{\ell^{k_{\ell}-2}}.$$
 Let $L\subset J'[\ell^{n-1}]$ be a subgroup of rank 2 such that
 $\ell^{n-2}L$ is maximal isotropic with respect to the Weil pairing and
 consider $Q_1,Q_2\in L\backslash G'$. Then~$\ell^{n-2}
 Q_1,\ell^{n-2}Q_2\in Ker~I^{\dagger}$. Since
 $T_{\ell^{n-1}}(I^{\dagger}(Q_1),I^{\dagger}(Q_2))\in
 \mu_{\ell^{k_{\ell}-4}}$, it follows that $T_{\ell^{n-1}}(Q_1,Q_2)\in
 \mu_{\ell^{k_{\ell}-3}}$. Hence
 $v_{\ell,\textrm{End}J'}(\pi)=v_{\ell,\textrm{End}J}(\pi)$ which
 contradicts the hypothesis that $I$ is descending.\\

Let $G\in \mathcal{W}$. By an argument similar to the one in Lemma~\ref{antisymmetry}, in order to determine the largest integer $k$ such that $T_{\ell^n}:G\times G\rightarrow \mu_{\ell^k}$ is surjective, it suffices to determine the largest $k$ such that all the self-pairings $T_{\ell^n}(P,P)$, with $P\in G$, are $\ell^k$-th roots of unity.
Let $G$ and $G'$ in $\mathcal{W}$ such that $\ell^{n-1}G=\ell^{n-1}G'$. First note that $P'\in G'$ can be written as $P'=P+L$, with $P\in G$ and $L\in J[\ell^{n-1}]$. Then by bilinearity
\begin{eqnarray*}
T_{\ell^n}(P',P')=T_{\ell^n}(P,P)(T_{\ell^n}(P,L)T_{\ell^n}(L,P))T_{\ell^n}(L,L).
\end{eqnarray*}
By Lemma~\ref{antisymmetry} and given that $L\in J[\ell^{n-1}]$, we have that $T_{\ell^n}(P',P')$ is a $\ell^{k_{\ell}}$-th primitive root of unity if and only if $T_{\ell^n}(P,P)$ is a $\ell^{k_{\ell}}$-th primitive root of unity. This implies that in order to compute $k_{\ell}$ it suffices to compute pairings over a set of representatives of $\mathcal{W}$ modulo the equivalence relation $G\sim G'$ if and only if $\ell^{n-1}G=\ell^{n-1}G'$.

Consequently, in order to find all kernels of horizontal isogenies we search, among subgroups $G\in \mathcal{W}$ (modulo the $\ell^{n-1}$-torsion), those for which the Tate pairing restricted to $G\times G$ maps to $\mu_{\ell^{k_{\ell,J-1}}}$.
 If $\{Q_1,Q_{2},Q_{-1},Q_{-2}\}$ is a symplectic basis for $J[\ell^n]$, then a subgroup of rank 2 generated by $\lambda_1Q_1+\lambda_{-1}Q_{-1}+\lambda_2Q_2+\lambda_{-2}Q_{-2}$ and $\lambda'_1Q_1+\lambda'_{-1}Q_{-1}+\lambda'_2Q_2+\lambda'_{-2}Q_{-2}$, with $\lambda_i,\lambda'_j\in \F_{\ell}$, $i,j\in \{-2,-1,1,2\}$, is maximal isotropic with respect to the Weil pairing if the following equation is satisfied
\begin{eqnarray}\label{lula}
\lambda_1\lambda'_{-1}-\lambda_{-1}\lambda'_1+\lambda_2\lambda'_{-2}-\lambda_{-2}\lambda_2=0.
 \end{eqnarray}
Moreover, this subgroup has degenerate Tate pairing if the following equations are satisfied
\begin{eqnarray}
\sum_{i,j\in \{1,2,-1,-2\}} \lambda_i\lambda_j\log T_{\ell^n}(Q_i,Q_j)&=&0 \bmod \ell^{n-k_{\ell}+1}\\
\sum_{i,j\in \{1,2,-1,-2\}} \lambda_i\lambda'_j\log T_{\ell^n}(Q_i,Q_j)&=&0 \bmod \ell^{n-k_{\ell}+1}\\
\sum_{i,j\in \{1,2,-1,-2\}} \lambda'_i\lambda'_j\log T_{\ell^n}(Q_i,Q_j)&=&0 \bmod \ell^{n-k_{\ell}+1}
\end{eqnarray}

\begin{example}
We consider the jacobian of the hyperelliptic curve
\begin{eqnarray*}
 y^2=5x^5+4x^4+98x^2+7x+2,
\end{eqnarray*}
defined over the finite field $\F_{127}$. The jacobian has maximal endomorphism ring at $5$ and $[\textrm{End}J:\Z[\pi,\bar{\pi}]]=50$. The ideal $(5)$ decomposes as $5=\mathfrak{a_1}\mathfrak{a_2}$ in $\mathcal{O}_K$. Hence there are two horizontal isogenies, which correspond to ideals $\mathfrak{a_1}$ and $\mathfrak{a_2}$ under the Shimura class group action. The 5-torsion is defined over an extension field of degree 8 of the field $\F_{127}$, that we denote $\F_{127}(t)$. Our computations with MAGMA found two subgroups of $J[5]$, maximal isotropic with respect to the Weil pairing and with degenerate 5-Tate pairing. For lack of space, we give here the Mumford coordinates of the generators of one of these subgroups.

\begin{eqnarray*}
&&(x^2 + (74t^7 + 25t^6 + 6t^5 + 110t^4 + 96t^3 + 75t^2 + 29t + 20)x\\ && +
39t^7 + 62t^6 + 77t^5 + 47t^4 + 9t^3 + 62t^2 + 97t + 15,\\ &&(116t^7 +
61t^6 + 13t^5 + 38t^4 + 70t^3 + 109t^2 + 62t + 71)x + 98t^7\\&& + 77t^6
 + 17t^5 + 76t^4 + 81t^3 + 5t^2 + 36t + 33)\\
&& (x^2 + (66t^7 + 89t^6+ 50t^5 + 124t^4 + 91t^3 + 102t^2 + 100t + 52)x\\&& + 119t^7 + 14t^6+ 126t^5 + 42t^4 + 42t^3 + 85t^2 + 12t + 77,\\&& (92t^7 + 90t^6 + 94t^5 +57t^4 + 59t^3 + 24t^2 + 72t + 11)x\\&& + 103t^7 + 16t^6 + 7t^5 + 111t^4+95t^3+79t^2+45t+34)
\end{eqnarray*}

\end{example}

\section{Complexity analysis}\label{Complexity}
In this section, we evaluate the complexity of Algorithm~\ref{LocallyMaximal} and compare its performance to that of the Freeman-Lauter algorithm. Note that for a fixed $\ell>2$, both algorithms perform computations in extension fields over which the $\ell^d$-torsion, for a certain $\ell^d$ dividing $[\mathcal{O}_K:\Z[\pi,\bar{\pi}]]$, is rational.

\paragraph{Checking locally maximal endomorphism rings.}
 In Freeman and Lauter's algorithm, in order to check if $\End(J)$ is locally maximal at $\ell$, for $\ell>2$, it suffices to check that $\sqrt{d}$ and $\eta$  are endomorphisms of $J$ (see~\cite[Lemma 6]{EisLau}). If $\pi=c_1+c_2\sqrt{d}+(c_3+c_4\sqrt{d})\eta$\footnote{Note that we cannot always write $\pi$ in this form, but if this is not case, we can always replace $\pi$ by $2^s\pi$, for some $s\in \Z$.} then we have
\begin{eqnarray}
2c_2\sqrt{d}&=&\pi+\bar{\pi}-2c_1\\
(4c_2(c_3^2-c_4^2d))\eta &=&(2c_2c_3-c_4(\pi+\bar{\pi}-2c_1))(\pi-\bar{\pi}).
\end{eqnarray}
Moreover, Eisentr\"ager and Lauter show that the index is $[\mathcal{O}_K:\Z[\pi,\bar{\pi}]]=2^sc_2(c_3^2-c_4^2d)$, for some $s\in \N$. Hence, for a fixed $\ell>2$ dividing the index $[\mathcal{O}_K:\Z[\pi,\bar{\pi}]]$, we need to consider an extension field over which $J[\ell^u]$ is defined, where $u$ is the $\ell$-adic valuation of the index. Meanwhile, Algorithm~\ref{LocallyMaximal} performs computations over the smallest extension field containing the $\ell$-torsion points. The degree of this extension field is smaller than $\ell^3$, by Proposition~\ref{extensionField}.

\vspace{0.3 cm}
\noindent
\textit{Notation.}
We denote by $r$ the degree of the smallest extension field $\F_{q^r}$ such that the $\ell$-torsion is $\F_{q^r}$-rational.\\

\noindent
We suppose that $\pi^r-1$ is exactly divisible by $\ell^n$.  First, we need to compute a basis for the $\ell^n$-torsion. We assume that the zeta function of $J/\F_{q^r}$ and the factorization $\#J(\F_{q^r})=\ell^sm$ are known in advance. We denote by $M(r)$ the cost of multiplication in an extension field of degree $r$. In order to compute the generators of $J[\ell^n]$, we use an algorithm implemented in AVIsogenies~\cite{AVISOGENIES}, which needs $O(M(r)(r\log q+\ell^n))$ operations in $\F_q$. We then compute a symplectic basis of $J[\ell^n]$, by using an algorithm similar to Gram\mbox{--}Schmidt orthogonalization.
 In order to compute $k_{\ell}$, we use the values of the Tate pairing $T_{\ell^n}(Q_i,Q_j)$ for $i,j\in \{1,-1,2,-2\}$. Computing the Tate pairing costs $O(M(r)(n\log \ell+r\log q))$ operations in $\F_q$, where the first term is the cost of Miller's algorithm and the second one is the cost for the final exponentiation.
 We conclude that the cost of Algorithm~\ref{LocallyMaximal} is $O(M(r)(r\log q++\ell^n+n\log \ell))$. The complexity of Freeman and Lauter's algorithm for endomorphism ring computation is dominated by the cost of computing the $\ell$-Sylow group of the Jacobian defined over the extension field containing the $\ell^u$-torsion, whose degree is $r\ell^{u-r}$ (by Proposition~\ref{moreTorsion}). The costs of the two algorithms are given in Table~\ref{Table1}.

\begin{table}[ht]
\centering
\caption{ \label{Table1} Cost for checking locally maximal endomorphism rings at $\ell$}
\begin{tabular}{|c|c|}
\hline
 Freeman and Lauter  & This work (Algorithm~\ref{LocallyMaximal}) \\
\hline
\multirow{3}{*}{$O(M(r+\ell^{u-r})(r\ell^{u-r}\log q+\ell^u))$}&\multirow{3}{*}{$O(M(r)(r\log q+\ell^n+n\log \ell))$}\\
 &\\
 & \\
\hline
\end{tabular}
\end{table}

\paragraph{Computing horizontal isogenies.} Both classical algorithms and our algorithm need to compute first a basis for the $\ell$-torsion. As stated before, this costs $O(rM(r)\log q)$. The classical algorithm (see~\cite[Algorithm VI.3.4]{Bisson}) computes subspaces which are invariant under the action of Frobenius. More precisely, this algorithm needs to compute the matrix of the Frobenius endomorphism (in $O(\ell^2)$ operations in $\F_{q^r}$ using a baby-step giant-step approach). We conclude that the overall complexity of this algorithm is $O(M(r)(r\log q+\ell^2))$. The method described in Section~\ref{HorizontalIsogenies} computes a symplectic basis of the $\ell^n$-torsion and solves a system of 4 homogenous equations of degree $2$, with coefficients in $\F_{\ell}$. The cost of solving this system is polynomial in $\ell$ and thus negligible ($\ell$ is small). Our method for horizontal isogeny computation has the same cost as Algorithm~\ref{LocallyMaximal}.

\section{Conclusion}

For an ordinary jacobian defined over a finite field, we have described a
relation between its endomorphism ring and some properties of the
$\ell$-Tate pairing. We deduced an efficient criterion for checking whether
the jacobian is locally maximal at $\ell$ and an algorithm computing
kernels of horizontal $(\ell,\ell)$-isogenies.

\section{Acknowledgements} This work was supported by the Direction
G\'en\'erale de l'Armement through the AMIGA project under contract
2010.60.055 and by the French Agence Nationale de la Recherche through the
CHIC project. The author thanks David Gruenewald and John Boxall for helpful discussions and is particularly indebted to Ben Smith for valuable comments and proofreading of previous versions of this manuscript.
\bibliographystyle{plain}
\bibliography{sorina1}

\section{Appendix A}\label{Annexe1}
We detail the proof of Lemma~\ref{divisors}.
\begin{proof}
(a) We can easily check that
\begin{eqnarray*}
f_{\ell^{n+1},\bar{D}_2}=(f_{\ell,\bar{D}_2})^{\ell^n}\cdot f_{\ell^n,D_2}.
\end{eqnarray*}
Note that these functions are $\F_q$-rational.
By evaluating them at $D_1$ and raising to the power $(q-1)/\ell^n$, we obtain the desired equality.
(b) Since $\textrm{div}\,(f_{\ell^{n+1},D_2})=\textrm{div}\,(f_{\ell^n,D_2}^{\ell})$,
we have $T_{\ell^{n+1}}^{\ell}(\bar{D}_1,\bar{D}_2)=T_{\ell^n}^{(\F_{q^{\ell}})}(\bar{D}_1,D_2)$,
where $T_{\ell^n}^{(\F_{q^{\ell}})}$ is the $\ell^n$-Tate pairing defined over $\F_{q^{\ell}}$.
We only need to show that
$$T_{\ell^n}^{(\F_{q^{\ell}})}(\bar{D}_1,D_2)=T_{\ell^n}(D_1,D_2)$$
Note that we have $\pi(\bar{D}_1)=\bar{D}_1+D_{\ell}$, where $D_{\ell}$ is a point of order $\ell$.
This implies that
$$\bar{D}_1+\pi(\bar{D}_1)+\pi^2(\bar{D}_1)+\ldots+\pi^{\ell-1}(\bar{D}_1)\sim \ell \bar{D}_1\sim D_1.$$
Hence we get
\begin{eqnarray*}
T_{\ell^n}^{(\F_{q^{\ell}})}(\bar{D}_1,D_2)&=&f_{\ell^n,D_2}(\bar{D}_1)^{\frac{(1+q+\ldots+q^{\ell-1})(q-1)}{\ell^n}}\\
&&=f_{\ell^n,D_2}(\bar{D}_1+\pi(\bar{D}_1)+\ldots+\pi^{\ell-1}(\bar{D}_1))^{\frac{(q-1)}{\ell^n}}.
\end{eqnarray*}
By applying Weil's reciprocity law, we obtain
\begin{eqnarray*}
T_{\ell^n}^{(\F_{q^{\ell}})}(\bar{D}_1,D_2)=f_{\ell^{n},D_2}(D_1)^{\frac{(q-1)}{\ell^n}}f(D_2)^{q-1},
\end{eqnarray*}
where $f$ is such that $\textrm{div}(f)=(\bar{D}_1)+(\pi(\bar{D}_1))+\ldots+(\pi(\bar{D}_1))-D_1$ and that $\textrm{supp}(f)\cap \textrm{supp}(D_2)=\emptyset$. Note that $f$ is $\F_q$-rational, so $f(D_2)^{q-1}=1$. This concludes the proof.
\end{proof}

\end{document}